\documentclass[11pt,reqno]{amsart}

\usepackage{import}
\usepackage{amsmath}
\usepackage{calc}
\PassOptionsToPackage{hyphens}{url}\usepackage[hidelinks,hypertexnames=false]{hyperref}
\usepackage{graphicx}

\usepackage[british]{babel}
\usepackage[noabbrev, nameinlink]{cleveref}
\usepackage{amsthm}
\usepackage{amssymb}
\usepackage{fullpage}
\usepackage{dsfont}
\usepackage{enumerate}
\usepackage{mathtools}
\usepackage{comment}
\usepackage{xcolor}
\usepackage{lineno}
\usepackage{centernot}
\usepackage{url}
\usepackage[normalem]{ulem}
\usepackage{thmtools}
\usepackage{thm-restate}

\newcommand{\eps}{\varepsilon}

\renewcommand{\le}{\leqslant}
\renewcommand{\ge}{\geqslant}

\def\N{{\mathbb{N}}}

\usepackage{stackengine}
\stackMath

\newtheorem{theorem}{Theorem}[section]

\newtheorem{lemma}[theorem]{Lemma}
\newtheorem{conjecture}[theorem]{Conjecture}

\crefformat{enumi}{#2item~(#1)#3}
\Crefformat{enumi}{#2Item~(#1)#3}
\crefmultiformat{enumi}{items~#2(#1)#3}%
  { and~#2(#1)#3}{, #2(#1)#3}{ and~#2(#1)#3}
\crefrangeformat{enumi}{items~#3(#1)#4--#5(#2)#6}

\crefname{subsection}{subsection}{subsections}

\title{A note on the maximum ratio between \\ chromatic number and clique number}

\begin{otherlanguage}{brazilian}
\author{Igor Araujo, Rafael Filipe, and Rafael Miyazaki}
    
\address{Department of Mathematics, University of Illinois Urbana-Champaign, Urbana, Illinois 61801, USA}
\email{igoraa2@illinois.edu}

\address{IMPA, Estrada Dona Castorina 110, Jardim Botanico, Rio de Janeiro, 22460-320, Brasil} \email{rafael.santos@impa.br}

\address{Department of Mathematics, Emory University, Atlanta, Georgia, USA}
\email{rafael.kazuhiro.miyazaki@emory.edu}

\thanks{During this work, Igor Araujo was partially supported by Parker Memorial Fellowship, Schark Fellowship, and NSF RTG DMS-1937241  and Rafael Filipe was supported by CNPq.}
\end{otherlanguage}

\usepackage[numbers,longnamesfirst]{natbib}

\date{}

\begin{document}
\begin{abstract}
    Let $f(n)$ be the maximum, over all graphs $G$ on $n$ vertices, of the ratio $\frac{\chi(G)}{\omega(G)}$, where $\chi(G)$ denotes the chromatic number of $G$ and $\omega(G)$ the clique number of $G$. In 1967, Erd\H{o}s showed that 
    \[ \Big( \frac{1}{4} +o(1) \Big) \frac{n}{(\log_2 n)^2} \le f(n) \le \big( 4+o(1) \big) \frac{n}{(\log_2 n)^2} .\]
    We show that 
    \[ f(n) \le \big(c+o(1)\big) \frac{n}{(\log_2 n)^2}\]
    for some $c<3.72$.
    This follows from recent improvements in the asymptotics of Ramsey numbers and is the first improvement in the asymptotics of $f(n)$ established by Erd\H{o}s.
\end{abstract}

\maketitle

\section{Introduction}\label{sec:intro}

Ramsey Theory is a fundamental area of Graph Theory concerned with the existence of unavoidable structures in large graphs. The \emph{Ramsey number} $R(s,t)$ is the minimum $n$ such that every red/blue edge-coloring of the complete graph on $n$ vertices contains either a clique on $s$ vertices colored red or a clique on $t$ vertices colored blue. 

While Ramsey~\cite{Ramsey1930} proved in 1930 that Ramsey numbers are finite, the first explicit upper bound was obtained by Erd\H{o}s and Szekeres~\cite{E-Sz} in 1935, who showed that 
\begin{equation} \label{eq:E-Sz}
    R(s,t) \le \binom{s+t-2}{s-1}.
\end{equation}
Note that in the \emph{diagonal} case, this shows that $R(k,k)\le 4^{k+o(k)}$. 
On the other hand, Erd\H{o}s~\cite{Ramsey.lower} established the lower bound $R(k,k)\ge 2^{k/2+o(k)}$ in 1947 in one of the first uses of the probabilistic method. Subsequent improvements on the lower bound have only refined it by a constant multiplicative factor (see, e.g., Spencer’s bound obtained via the Lov\'asz Local Lemma~\cite{Spencer.LLL}). Since then, one of the major problems in modern combinatorics has been to determine the correct asymptotic behavior of the diagonal Ramsey number $R(k,k)$. More specifically, one could ask whether the limit
\begin{equation}\label{limit1}
    \lim_{k \to \infty} \frac{\log(R(k,k))}{k}
\end{equation} 
exists and, if so, what is its value\footnote{This question appears as Problem 1 in~\cite{chung}, as Problem 77 in~\url{https://www.erdosproblems.com/77}, and on the webpage~\url{https://mathweb.ucsd.edu/~erdosproblems/}.}.

In 1967, Erd\H{o}s~\cite{erdos} asked a seemingly unrelated question. Define $f(n)$ to be the maximum, over all graphs $G$ on $n$ vertices, of the ratio $\frac{\chi(G)}{\omega(G)}$, where $\chi(G)$ denotes the chromatic number of $G$ and $\omega(G)$ the clique number of $G$. Formally, we have
\begin{equation*}
    f(n)\coloneq \max \Big\{ \ \frac{\chi(G)}{\omega(G)} \ : 
    \ G \text{ is a graph on } n \text{ vertices } \Big\}.
\end{equation*}

Erd\H{o}s showed that $f(n) = \Theta(n/(\log n)^2)$ and asked\footnote{This question also appears in~\cite{erdos2}, as Problem 627 in~\url{https://www.erdosproblems.com/627}, as Problem 53 in~\cite{chung}, and on the webpage~\url{https://mathweb.ucsd.edu/~erdosproblems/}.} whether the following limit exists: 
\begin{equation}\label{limit2}
    \lim_{n \to \infty} \frac{f(n)}{n/(\log n)^2}.
\end{equation} 
In this note, we establish a connection between these two problems. 
The upper bound on $f(n)$ from~\cite{erdos} combines~\eqref{eq:E-Sz} with the observation that, if $\binom{s+t}{t} \ge n$, then 
$$st \ge \lfloor{k/2}\rfloor\lfloor{(k+1)/2}\rfloor,$$
where $k$ is the smallest integer satisfying $\binom{k}{\lfloor{k/2}\rfloor} \ge n$.
This motivates Conjecture~\ref{conj:weak.mult.RDC} below as a natural bridge to extend Erd\H{o}s' proof. 

\begin{conjecture}\label{conj:weak.mult.RDC}
    For every $s, t, k \in \N$ such that $st \le k^2$ we have that    
    \begin{equation*}    
        R(s,t) \le R(k, k).
    \end{equation*}
\end{conjecture}

We were unable to find a previous statement of Conjecture~\ref{conj:weak.mult.RDC} in the literature, but it is closely related to the Diagonal Conjecture (Conjecture~\ref{conj:RDC}; see Section~\ref{sec:conj} for a more detailed discussion). 
The following theorem establishes a connection between the limits in \eqref{limit1} and \eqref{limit2}, should they exist, under the assumption of Conjecture~\ref{conj:weak.mult.RDC}. Throughout this note, all logarithms are in base 2.

\begin{theorem}\label{thm:diagonal_to_ratio}
    If Conjecture~\ref{conj:weak.mult.RDC} holds, and $\lim\limits_{k \to \infty} \frac{\log(R(k,k))}{k}$ exists and is equal to $\ell$, then
    \begin{equation*}     
        f(n)=\big(\ell^2+o(1)\big)\frac{n}{(\log n)^2} .
    \end{equation*} 
\end{theorem}

In~\cite{erdos}, Erd\H{o}s mentioned that by their method, it would be easy to prove that 
\begin{equation*}
    \big(c+o(1)\big) \frac{n}{(\log n)^2} \le f(n) 
    \le \big(C+o(1) \big)  \frac{n}{(\log n)^2},
\end{equation*}
for constants $c=\frac{1}{4}$ and $C=1$. However, this appears to be a typographical error\footnote{Indeed, if $f(n) \le (1+o(1))\frac{n}{(\log n)^2}$ is true, then Theorem~\ref{thm:upper} would imply that $R(k,k)\le 2^{k+o(k)}$.}, as a careful application of their method yields the constants $c=\frac{1}{4}$ and $C=4$. 

Aided by recent breakthrough developments in Ramsey theory, we can prove the following.

\begin{theorem}\label{thm:g_upper}
    The function $f(n)$ satisfies
\begin{equation}\label{eq:thm.bound.f.1}
 f(n) \le \big( 3.71943+o(1) \big) \frac{n}{(\log n)^2}.
\end{equation}
Moreover, if Conjecture~\ref{conj:weak.mult.RDC} holds, then
 \begin{equation}\label{eq:thm.bound.f.2}
    f(n) \le \big( 3.70831 + o(1) \big) \frac{n}{(\log n)^2}.
\end{equation}
\end{theorem}

This upper bound on $f(n)$ is the first improvement in the asymptotics of $f(n)$ since 1967.   

\subsection{Conjecture~\ref{conj:weak.mult.RDC} and the Diagonal Conjecture} \label{sec:conj}

We note that Conjecture~\ref{conj:weak.mult.RDC} is closely related to the following conjecture, which is named the \emph{Diagonal Conjecture} in~\cite{LIANG2019195}. 

\begin{conjecture}[Diagonal Conjecture (DC)]\label{conj:RDC}
For every $s_1\le s_2\le t_2\le t_1$ such that $s_1+t_1 \le s_2+t_2$ we have that 
    \begin{equation*}    
    R(s_1, t_1) \le R(s_2, t_2).
    \end{equation*}
\end{conjecture}

It is widely believed that the Diagonal Conjecture is very difficult to prove. Even in the case $R(t-1,t+1) \le R(t,t)$, there is no relevant progress. Observe that Conjecture~\ref{conj:weak.mult.RDC} can be stated as a weak version of the following conjecture.

\begin{conjecture}[Multiplicative form of DC] \label{conj:mult.RDC}
    For every $s_1\le s_2\le t_2\le t_1$ such that $s_1t_1 \le s_2t_2$ we have that
    \begin{equation*}    
    R(s_1, t_1) \le R(s_2, t_2).
    \end{equation*}
\end{conjecture}

It is easy to see that if $s_1 \le s_2 \le t_2 \le t_1$ and $s_1t_1 \le s_2t_2$, then $s_1+t_1 \le s_2+t_2$ and therefore Conjecture~\ref{conj:mult.RDC} implies Conjecture~\ref{conj:RDC}.

\subsection*{Organization of the paper} In Section~\ref{sec:proof_1.2}, we introduce some notation used throughout the paper and state bounds on $f(n)$, namely Theorems~\ref{thm:upper} and~\ref{thm:lower}, establishing a connection between the asymptotic behavior of $f(n)$ and the Ramsey numbers $R(s,t)$. The proof of Theorem~\ref{thm:diagonal_to_ratio}, which is an immediate consequence of these results, is also given in Section~\ref{sec:proof_1.2}. In Section~\ref{subsec:numerics}, we establish the numerical bounds of Theorem~\ref{thm:g_upper}.

\section{Proof of Theorem~\ref{thm:diagonal_to_ratio}}
\label{sec:proof_1.2}
In this section, we establish bounds sufficient to prove Theorem~\ref{thm:diagonal_to_ratio}. Let us define 
\begin{equation*}
g(n) \coloneqq \frac{(\log n)^2}{n} f(n),\quad
M\coloneq \limsup_{t \to \infty } \; \max_{s\le t} \ \frac{\log(R(s,t))}{\sqrt{st \ }},
\end{equation*}
\begin{equation*}
L\coloneqq \liminf\limits_{k \to \infty} \; \frac{\log (R(k,k))}{k}
\quad \text{ and } \quad D\coloneq\limsup_{k \to \infty} \;\frac{\log (R(k,k))}{k}.
\end{equation*}

The following theorems establish more precisely the connection between limits \eqref{limit1} and \eqref{limit2}. The first result gives the exactly value of $\limsup g(n)$ in terms of the Ramsey numbers.

\begin{theorem}\label{thm:upper}
    The function $g(n)$ satisfies
    \begin{equation*}
         \limsup_{n \to \infty} g(n) = M^2 .
    \end{equation*}
\end{theorem}

The second one establishes a lower bound on $\liminf g(n)$ in terms of the Ramsey numbers.

\begin{theorem}\label{thm:lower}
    The function $g(n)$ satisfies
    \begin{equation*}
         \liminf_{n \to \infty} g(n) \ge L^2.
    \end{equation*}
\end{theorem}

Before proceeding to the proof of these results, let us show how to derive Theorem~\ref{thm:diagonal_to_ratio}.

\begin{proof}[Proof of Theorem~\ref{thm:diagonal_to_ratio}]
The assumption that $\lim\frac{\log R(k,k)}{k}$ exists and equals $\ell$ implies $L=D=\ell$. Thus, by Theorems~\ref{thm:upper}~and~\ref{thm:lower}, it suffices to prove that $D=M$. To this end, observe that restricting the maximum to $s=t$ in the definition of $M$ yields
\begin{equation}\label{ineq:limsups}
D=\limsup_{k \to \infty} \; \frac{\log (R(k,k))}{k} \le \limsup_{t \to \infty } \; \max_{s\le t} \ \frac{\log(R(s,t))}{\sqrt{st \ }} = M. 
\end{equation}

On the other hand, if $0 < s \le t$ and $k$ is the positive integer such that $(k-1)^2 < st \le k^2$, Conjecture~\ref{conj:weak.mult.RDC} implies that 
\begin{equation*}
    R(s,t) \le R(k,k) \le 2^{(D+o_k(1))k}.
\end{equation*}

Therefore, we have
\begin{equation*}
    \frac{\log(R(s,t))}{\sqrt{st \ }} \le \frac{\log(R(k,k))}{k-1} \le D +o_k(1) = D +o_t(1),
\end{equation*}
which implies $M \le D$. Combined with \eqref{ineq:limsups}, we conclude that $D = M$.
\end{proof}

We now proceed to prove Theorems~\ref{thm:upper} and~\ref{thm:lower}, starting with Theorem~\ref{thm:lower}.

\begin{proof}[Proof of \Cref{thm:lower}]
For each $n\in \N$, let $k_n\in \N$ be such that $R(k_n,k_n) \le n < R(k_n+1,k_n+1)$.
By the definition of $L$, observe that 
\begin{equation}\label{eq:bound.logn}
    \log n \ge \big(L+o_n(1)\big)k_n.
\end{equation}

Let $G$ be a graph on $n$ vertices with no clique or independent set of size $k_n+1$. In other words, $\alpha(G) \le k_n$ and $\omega(G) \le k_n$. As $\chi(G) \ge n/\alpha(G)$, we conclude that
\begin{equation*}
    g(n)=\frac{f(n)}{n/(\log n)^2} \ge \frac{\chi(G)}{\omega(G)}\cdot\frac{(\log n)^2}{n} \ge 
\frac{(\log n)^2}{\alpha(G)\omega(G)}\ge
\frac{(\log n)^2}{k_n^2} \ge
L^2 + o_n(1),
\end{equation*}
where the last inequality is true by \eqref{eq:bound.logn}. Hence, $\liminf g(n) \ge L^2$.
\end{proof}

Now we focus on proving Theorem~\ref{thm:upper}. The proof relies on the following bound on the chromatic number of a graph from~\cite{erdos}, which is obtained by greedily picking maximum independent sets as color classes until few vertices remain, with each remaining vertex receiving a new color.

\begin{lemma}\label{lem:erdos}
    If $\alpha(G') \ge r$ for every subgraph $G' \subset G$ with at least $m_0$ vertices, then 
    \begin{equation*}
        \chi(G) \le \frac{n}{r} + m_0 .
    \end{equation*} 
\end{lemma}

With this lemma in hand, the proof is straightforward.

\begin{proof}[Proof of \Cref{thm:upper}]
Let $G$ be a graph on $n$ vertices and $m \in \mathbb{N}$ be such that 
\begin{equation*}
    n \ge m \ge \frac{n}{(\log n)^3}.
\end{equation*}

For $1\le s \le t$ such that $m < R(s+1, t+1)$, we have that either $s< \log t$ and 
\begin{equation*}
    R(s+1, t+1) \le \binom{s+t}{s} = 2^{o_t\left( \sqrt{st } \right)},
\end{equation*}
or $s \ge \log t$ and 
\begin{equation*}
    R(s+1,t+1) \le 2^{(M+o_{t}(1)) \sqrt{st}},
\end{equation*}
where the last inequality follows from the definition of $M$ and the fact that $s$ goes to infinity with~$t$. Thus, in any case, we obtain
\begin{equation}\label{eq:bound.logm}
    \log m \le \big(M+o_m(1)\big) \sqrt{st}.
\end{equation}

Notice that any graph $H$ on $m$ vertices provides a red/blue edge-coloring of $K_m$ implying that $m < R(\omega(H)+1, \alpha(H)+1)$. Thus, by \eqref{eq:bound.logm}, for every subgraph $G'$ of $G$ on $m$ vertices, we have
\begin{equation*}
     \omega(G') \cdot \alpha(G') \ge 
\Big( \frac{1}{M^2} + o_m(1) \Big) (\log m)^2. 
\end{equation*}
Since $\omega(G') \le \omega(G)$, we obtain that \begin{equation*}
    \alpha(G') \ge
\Big( \frac{1}{M^2} + o_m(1) \Big) \frac{(\log m)^2}{\omega(G)}.
\end{equation*} 
Hence, together with the fact that $\log m = \big(1+o_n(1)\big) \log n$, \Cref{lem:erdos} yields
\begin{equation*}
    \chi(G) \le \big(M^2 +o_n(1)\big)\frac{n\cdot\omega(G)}{(\log n)^2} + \frac{n}{(\log n)^3}.
\end{equation*}
Then, we conclude that
\begin{equation}\label{eq:UpperBoundM2}
    \frac{\chi(G)}{\omega(G)} \le \big(M^2 +o_n(1)\big)\frac{n}{(\log n)^2}.
\end{equation}

Moreover, we know that there exist an increasing sequence $\{t_i\}_{i \in \N}$ and a sequence $\{s_i\}_{i \in \N}$ such that $s_i \le t_i$ for every $i \in \N$, and
\begin{equation*}
    \lim\limits_{i \to \infty} \frac{\log (R(s_i,t_i))}{\sqrt{s_it_i\ }}=M.
\end{equation*}

Then, for $n_i \coloneqq R(s_i+1, t_i+1)-1$, we have $(\log n_i)^2 \ge s_it_i\big(M^2+o_i(1)\big)$, since $n_i > R(s_i,t_i)$. The definition of $n_i$ also implies that there is a graph $G_i$ on $n_i$ vertices satisfying $\alpha(G_i) \le s_i$ and $\omega(G_i) \le t_i$. Thus we have
\begin{equation*}
     g(n_i)=\frac{f(n_i)}{n_i/(\log n_i)^2} \ge \frac{\chi(G_i)}{\omega(G_i)}\cdot\frac{(\log n_i)^2}{n_i} \ge \frac{(\log n_i)^2}{\alpha(G_i)\omega(G_i)}\ge\frac{(\log n_i)^2}{s_it_i} \ge M^2 + o_i(1),
\end{equation*}
which together with \eqref{eq:UpperBoundM2} implies that $\limsup g(n) = M^2$.\qedhere
\end{proof}

\section{Proof of Theorem~\ref{thm:g_upper}} \label{subsec:numerics}
In this section, we prove Theorem~\ref{thm:g_upper}. The currently best known upper bound for $R(k,k)$ is
\begin{equation}\label{eq:diagonal_upper_bound}
    R(k,k) \le \big( 4e^{-0.14e^{-1}} \big)^{k+o(k)}.
\end{equation} 
This bound follows from the recent breakthrough result by Campos, Griffiths, Morris, and Sahasrabudhe~\cite{CGMS}, along with its improvement by Gupta, Ndiaye, Norin, and Wei~\cite{GNNW}. Precisely, they showed that
\begin{equation}\label{eq:CGMS}
R(s, t) \le e^{-\delta s+o(t)} \binom{s+t}{s}
\end{equation}
for $s\le t$, where $\delta = 0.14e^{-1}$.
With this in hand, we can now prove Theorem~\ref{thm:g_upper}.

\begin{proof}[Proof of Theorem~\ref{thm:g_upper}]
As seen in the proof of Theorem~\ref{thm:diagonal_to_ratio}, if Conjecture~\ref{conj:weak.mult.RDC} holds, then $D=M$. Hence, Theorem~\ref{thm:upper} and \eqref{eq:diagonal_upper_bound} imply~\eqref{eq:thm.bound.f.2}. Indeed,
\begin{equation*}
    f(n) \le \big( ( \log 4e^{-0.14e^{-1}} )^2 + o(1) \big) \frac{n}{(\log n)^2} < \big( 3.70831 + o(1) \big) \frac{n}{(\log n)^2}.
\end{equation*}

Thus, we now focus on proving~\eqref{eq:thm.bound.f.1}. For $s\le t$, let $x = \frac{s}{s+t}$.
First, note that $\log R(s,t) = o(\sqrt{st})$ when $s=o(t)$. Then, we assume that $x \in [\eps, 1/2]$ for some $\eps > 0$. Observe that \eqref{eq:CGMS} implies
\begin{equation*}
       \frac{\log(R(s,t))}{\sqrt{st \ }} \le \frac{\log \binom{s+t}{s} - \delta s \log e+o_t(t)}{\sqrt{st}} = \frac{\log \binom{s/x}{s}-\delta s \log e}{s\sqrt{\frac{1-x}{x}}}+o_t(1).
\end{equation*}

Since $\log \binom{n}{\alpha n} \le H(\alpha) \cdot n$ for $\alpha\in (0,1)$, where $H(x) = -x \log x - (1-x) \log (1-x)$ is the binary entropy function, we obtain 
\begin{equation*}
    \frac{\log \binom{s/x}{s} - \delta s \log e}{s\sqrt{\frac{1-x}{x}}}\le\frac{H(x) \cdot \frac{s}{x}  - \delta s \log e}{s\sqrt{\frac{1-x}{x}}} =
\frac{H(x)  - \delta x \log e}{\sqrt{x(1-x)}}.
\end{equation*}

Therefore, we have that
\begin{equation*}
 \frac{\log(R(s,t))}{\sqrt{st \ }} \le \frac{ H(x)  - \delta x \log e}{\sqrt{x(1-x)}}+o_t(1).
\end{equation*}

Define $\varphi(x) \coloneq \frac{H(x)-\delta x \log e}{\sqrt{x(1-x)}}$. By differentiating, the maximum of $\varphi(x)$ is achieved for some $x \in (0,1/2]$ satisfying 
\[(1-x)\log(1-x)-x\log(x) - \delta x \log e = 0.\]
For such $x$, since $\delta = 0.14e^{-1}$, we have that $\varphi(x)^2 < 3.71943$ and, by Theorem~\ref{thm:upper}, we conclude 
\begin{equation*}
    \limsup_{n \to \infty} g(n) = M^2  \le 
    \max_{x \in (0, 1/2]} \varphi(x)^2 < 3.71943.  \qedhere
\end{equation*}
\end{proof}

\section*{Acknowledgements}

The authors thank Rob Morris for his helpful suggestions that heavily improved the presentation of this paper.

\bigskip
\bibliographystyle{abbrv}
\def\bibfont{\footnotesize}
\bibliography{refs}
\end{document}